\documentclass[joc]{ipart}

\RequirePackage[OT1]{fontenc}
\RequirePackage{amsthm,amsmath, latexsym, amssymb}
\RequirePackage[numbers,square]{natbib}
\RequirePackage[colorlinks]{hyperref}
\RequirePackage{tikz}

\pubyear{0000}
\volume{0}
\issue{0}
\firstpage{1}
\lastpage{8}
\arxiv{http://arxiv.org/abs/1804.07879}

\startlocaldefs
\theoremstyle{plain}

\newtheorem{theorem}{Theorem}[section]
\newtheorem{proposition}[theorem]{Proposition}
\newtheorem{corollary}[theorem]{Corollary}

\newtheorem{example}[theorem]{Example}

\newtheorem{defn}[theorem]{Definition}

\newcommand{\inv}{{\mathrm {inv}}}

\newcommand{\Stir}{{\mathrm {Stir}}}
\newcommand{\Hilb}{{\mathrm {Hilb}}}

\newcommand{\coinv}{{\mathrm {coinv}}}

\newcommand{\code}{{\mathtt {code}}}

\newcommand{\st}{{\mathrm {st}}}
\newcommand{\Mat}{{\mathrm {Mat}}}
\newcommand{\conv}{{\mathrm {conv}}}
\newcommand{\PM}{{\mathrm {PM}}}
\newcommand{\initial}{{\mathrm {in}}}

\newcommand{\symm}{{\mathfrak{S}}}

\newcommand{\CC}{{\mathbb {C}}}
\newcommand{\QQ}{{\mathbb {Q}}}
\newcommand{\ZZ}{{\mathbb {Z}}}

\newcommand{\OP}{{\mathcal{OP}}}
\newcommand{\PP}{{\mathbb{P}}}

\newcommand{\VVV}{{\mathcal{V}}}
\newcommand{\WWW}{{\mathcal{W}}}
\newcommand{\UUU}{{\mathcal{U}}}
\newcommand{\CCC}{{\mathcal{C}}}

\newcommand{\MMM}{{\mathcal{M}}}

\newcommand{\xx}{{\mathbf {x}}}
\newcommand{\II}{{\mathbf {I}}}

\newcommand{\TT}{{\mathbf {T}}}

\newcommand{\neglex}{{\texttt {neglex}}}
\endlocaldefs

\begin{document}

\begin{frontmatter}
\title{Line configurations and $r$-Stirling partitions}
\runtitle{Line configurations and $r$-Stirling partitions}

\begin{aug}
\author{\fnms{Brendon}
\snm{Rhoades}\ead[label=e1]{bprhoades@math.ucsd.edu}}
\address{Department of Mathematics\\
University of California, San Diego\\
La Jolla, CA, 92093, USA\\
\printead{e1}}
\and
\author{\fnms{Andrew Timothy} \snm{Wilson}\ead[label=e2]{andwils2@pdx.edu}}
\address{Department of Mathematics\\
Portland State University\\
Portland, OR, 97201, USA\\
\printead{e2}}

\affiliation{Some University and Another University}
\end{aug}

\begin{abstract}
A set partition of $[n] := \{1, 2, \dots, n \}$ is called {\em $r$-Stirling} if the numbers
$1, 2, \dots, r$ belong to distinct blocks.
Haglund, Rhoades, and Shimozono constructed a graded ring $R_{n,k}$ depending on two 
positive integers $k \leq n$ whose algebraic properties are governed by
the combinatorics of ordered 
set partitions of $[n]$ with $k$ blocks.  We introduce a variant $R_{n,k}^{(r)}$ of 
this quotient for ordered $r$-Stirling partitions which depends on three integers 
$r \leq k \leq n$.  We describe the standard monomial basis of $R_{n,k}^{(r)}$ and use 
the combinatorial notion of the {\em coinversion code} of an ordered set partition
to reprove and generalize some results of Haglund et.\ al.\ in a more direct way.  
Furthermore, we introduce a variety $X_{n,k}^{(r)}$ of line configurations whose 
cohomology is presented as the integral form of $R_{n,k}^{(r)}$, generalizing
results of Pawlowski and Rhoades.
\end{abstract}

\received{\smonth{1} \sday{1}, \syear{0000}}

%\tableofcontents

\end{frontmatter}

\section{Introduction}
\label{Introduction}

Given two  integers $r \leq n$, a set partition of $[n] := \{1, 2, \dots, n \}$ is 
called {\em $r$-Stirling} if the first $r$ letters $1, 2, \dots, r$ lie in distinct blocks.
The {\em $r$-Stirling number (of the second kind)}
$\Stir_{n,k}^{(r)}$ counts $r$-Stirling partitions of $[n]$ with $k$ blocks.
An {\em ordered $r$-Stirling partition} is an $r$-Stirling partition
$\sigma = (B_1 \mid \cdots \mid B_k)$ equipped with a total order on its blocks.
We let $\OP_{n,k}^{(r)}$ denote the family of ordered $r$-Stirling partitions of $[n]$ with 
$k$ blocks;
these are counted by
$|\OP_{n,k}^{(r)} | = k! \cdot \Stir^{(r)}_{n,k}$.

An example element of $\OP_{7,4}^{(3)}$ is 
$(2 \,  6  \mid 5 \mid 1 \, 7  \mid 3 \, 4 )$.
On the other hand, the ordered set partition $(4 \, 5 \mid 2 \mid 1 \, 3 \, 6 \mid 7)$
fails to be $3$-Stirling since $1$ and $3$ belong to the same block.
The symmetric group $S_n$ acts on ordered set partitions of $[n]$ by letter permutation.
Although $\OP_{n,k}^{(r)}$ is not closed under the full action of $S_n$, it does carry
an action of the parabolic subgroup $S_r \times S_{n-r}$.

When $r = k = n$, an element of $\OP_{n,n}^{(n)}$ is just a permutation in  $S_n$.
The combinatorics of the symmetric group $S_n$ is well-known to govern 
both the algebraic structure of the {\em coinvariant ring} $R_n$
and the geometric structure of the {\em flag variety}
$\mathcal{F \ell}(n)$.  

In the case $r = 0$ where 
$\OP_{n,k} := \OP_{n,k}^{(0)}$ is the collection of $k$-block ordered set partitions 
of $[n]$, the 
 {\em Delta Conjecture} \cite{HRW}
 in the theory of Macdonald polynomials
motivated the definition and study of a generalized coinvariant ring
$R_{n,k}$ \cite{HRS}
and a generalization $X_{n,k}$ of the flag variety \cite{PR}  which specialize to their 
classical counterparts when $k = n$.  The algebraic properties of 
$R_{n,k}$ and the geometric properties of $X_{n,k}$ are governed by 
combinatorial properties of ordered set partitions in $\OP_{n,k}$.

At a workshop in Montr\'eal in the Summer of 2017, Jeff Remmel asked the authors if it
was possible to extend this theory to encapsulate ordered $r$-Stirling partitions;
in this paper we do exactly that.
We consider a quotient ring $R_{n,k}^{(r)}$ and a variety 
$X_{n,k}^{(r)}$ whose properties are controlled by the combinatorics of $\OP_{n,k}^{(r)}$.
The quotient $R_{n,k}^{(r)}$ of $\QQ[\xx_n] := \QQ[x_1, \dots, x_n]$
(together with its companion quotient
$S_{n,k}^{(r)}$ of $\ZZ[\xx_n] := \ZZ[x_1, \dots, x_n]$) is defined as follows.  
If $\xx_m = (x_1, \dots, x_m)$ is a list of variables and $d \geq 0$, we recall the {\em elementary}
and {\em homogeneous} symmetric polynomials of degree $d$ in the variable 
set $\xx_m$:
\begin{align}
e_d(\xx_m) &:= \sum_{1 \leq i_1 < \cdots < i_d \leq m} x_{i_1} \cdots x_{i_d}, \\
h_d(\xx_m) &:= \sum_{1 \leq i_1 \leq \cdots \leq i_d \leq m} x_{i_1} \cdots x_{i_d}.
\end{align}

\begin{defn}
\label{main-definition}
For $r \leq k \leq n$, let $I_{n,k}^{(r)} \subseteq \QQ[\xx_n]$ be the ideal
\begin{equation}
I_{n,k}^{(r)} := \left \langle
\begin{array}{c}
x_1^k, x_2^k, \dots, x_n^k, \\
e_n(\xx_n), e_{n-1}(\xx_n), \dots, e_{n-k+1}(\xx_n), \\
h_{k-r+1}(\xx_r), h_{k-r+2}(\xx_r), \dots, h_k(\xx_r)
\end{array}
\right \rangle
\end{equation}
and let $R_{n,k}^{(r)}$ be the corresponding quotient ring:
\begin{equation}
R_{n,k}^{(r)} := \QQ[\xx_n]/I_{n,k}^{(r)}.
\end{equation}
Furthermore, let $J_{n,k}^{(r)} \subseteq \ZZ[\xx_n]$ be the 
ideal in $\ZZ[\xx_n]$ with the same generating set as $I_{n,k}^{(r)}$ and 
let $S_{n,k}^{(r)} = \ZZ[\xx_n]/J_{n,k}^{(r)}$ be the corresponding
quotient.
\end{defn}

When $r = k = n$, the ideal $I_n := I_{n,n}^{(n)}$ is just the classical {\em invariant ideal}
$\langle e_1(\xx_n), e_2(\xx_n), \dots, e_n(\xx_n) \rangle$ generated by the $n$
elementary symmetric polynomials.  When $r = 0$, the ideal
$I_{n,k} := I_{n,k}^{(0)}$ is precisely the ideal considered in \cite{HRS}, and 
its companion ideal $J_{n,k} := J_{n,k}^{(0)}$ over the ring of integers 
was considered in \cite{PR}.

The quotient ring $S_{n,k}^{(r)}$ will be shown to calculate the 
cohomology (singular, with coefficients in $\ZZ$) of a natural space $X_{n,k}^{(r)}$
whose geometry is governed by the combinatorics of $\OP_{n,k}^{(r)}$.
Let $\PP^{k-1}$ be the complex projective space of lines through the origin in
$\CC^k$, so that $(\PP^{k-1})^n$ is the complex algebraic variety of all $n$-tuples 
$(\ell_1, \dots, \ell_n)$
of lines through 
the origin in $\CC^k$.  We consider the following family of line configurations.

\begin{defn}
\label{variety-definition}
Let $r \leq k \leq n$ and define a subset $X_{n,k}^{(r)} \subseteq (\PP^{k-1})^n$ by
\begin{equation}
X_{n,k}^{(r)} := \left\{ (\ell_1, \ell_2, \dots, \ell_n) \in (\PP^{k-1})^n \,:\,
\begin{array}{c}
\ell_1 + \ell_2 + \cdots + \ell_n = \CC^k \text{ and } \\
\dim(\ell_1 + \ell_2 + \cdots + \ell_r) = r \end{array} \right \}.
\end{equation}
\end{defn}

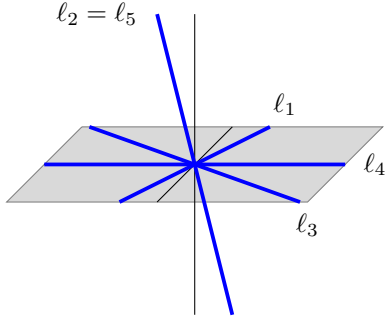
\begin{figure}
\begin{tikzpicture}

\draw [gray, fill=gray!30] (-1.5,0.5) -- (2.5,0.5) -- (1.5,-0.5) -- (-2.5,-0.5) -- (-1.5,0.5);

\draw  (0,-2) -- (0,2);

\draw (-2,0) -- (2,0);

\draw (0.5,0.5) -- (-0.5,-0.5);

\draw [line width = 0.5mm, blue] (1,0.5) -- (-1,-0.5);

\draw [line width = 0.5mm, blue] (-1.4,0.5) -- (1.4,-0.5);

\draw [line width = 0.5mm, blue] (0.5,-2) -- (-0.5,2);

\draw [line width = 0.5mm, blue] (-2,0) -- (2,0);

\node at (-1.3,2) {$\ell_2 = \ell_5$};

\node at (1.2,0.8) {$\ell_1$};

\node at (2.4,0) {$\ell_4$};

\node at (1.5,-0.8) {$\ell_3$};
\end{tikzpicture}
\caption{A point in $X_{5,3}^{(2)}$.}
\label{point}
\end{figure}

A typical point in $X_{n,k}^{(r)}$ is an $n$-tuple of lines 
$(\ell_1, \dots, \ell_n)$ through the origin in $\CC^k$ such that these lines span 
$\CC^k$ and such that the first $r$ of these lines are linearly independent.
An example of such a line configuration in $X_{5,3}^{(2)}$ is shown in Figure~\ref{point};
the first two lines $\ell_1$ and $\ell_2$ are linearly independent, and the five lines
$\ell_1, \dots, \ell_5$ together span $\CC^3$.

The product group $S_r \times S_{n-r}$ acts on $X_{n,k}^{(r)}$ by line permutation.
The set $X_{n,k}^{(r)}$ is a Zariski open subset of $(\PP^{k-1})^n$ and is therefore
both a variety and a smooth complex manifold.

When $r = k = n$, the space $X_{n,k}^{(r)}$ may be identified with the quotient $G/T$,
where $G = GL_n(\CC)$ is the group of invertible $n \times n$ complex matrices and
$T \subseteq G$ is the diagonal torus. 
 If $B \subseteq G$ is the Borel subgroup of upper triangular
matrices, the quotient $G/B$ is the classical flag variety $\mathcal{F \ell}(n)$
of type A$_{n-1}$ and
the canonical projection $G/T \twoheadrightarrow G/B$ is a homotopy equivalence.
When $r = 0$, the space $X_{n,k} := X_{n,k}^{(0)}$ of $n$-tuples of lines spanning 
$\CC^k$ was defined and studied by Pawlowski and Rhoades as an extension of the 
flag variety \cite{PR}.

The remainder of the paper is organized as follows.  In {\bf Section~\ref{Coinversion}}
we will introduce a new statistic on an ordered set partition $\sigma$: the {\em coinversion code}
$\code(\sigma)$.  This will allow us to read off the standard monomial basis of the quotient
ring $R_{n,k}^{(r)}$ directly from the combinatorics of $\OP_{n,k}^{(r)}$, both extending
and making more combinatorial the results regarding $R_{n,k}$ in \cite{HRS}.
In {\bf Section~\ref{Line}} we will study the space of line configurations
$X_{n,k}^{(r)}$ and prove that $H^{\bullet}(X_{n,k}^{(r)}) = S_{n,k}^{(r)}$.
We will also describe an affine paving of $X_{n,k}^{(r)}$ with cells indexed by 
partitions in $\OP_{n,k}^{(r)}$, together with formulas for the representatives of
the closures of these 
cells in cohomology.

\section{Coinversion codes and standard bases}
\label{Coinversion}

Recall that an {\em inversion} of a permutation $w  \in S_n$ is a 
pair $1 \leq i < j \leq n$ such that $i$ appears to the right of $j$ in the one-line notation
$w = w_1 \dots w_n$, so that the inversions of $231 \in S_3$ are the pairs 
$(1,2)$ and $(1,3)$.  Extending this notion to ordered set partitions, if 
$\sigma = (B_1 \mid \cdots \mid B_k)$ is an ordered set partition of $[n]$ with $k$ blocks,
a pair $1 \leq i < j \leq n$ is said to be an {\em inversion} of $\sigma$ if
\begin{itemize}
\item the block of $i$ is strictly to the right of the block of $j$ in $\sigma$, and
\item the letter $i$ is minimal in its block.
\end{itemize}
We let $\inv(\sigma)$ be the number of inversions of $\sigma$, so that if 
$\sigma = (2 5 \mid 1 \mid 3 4) \in \OP_{5,3}$ the inversion pairs are
$(1,2), (1,5),$ and $(3,5)$ so that $\inv(\sigma) = 3$.

We will not be interested in the statistic $\inv$ itself, but rather its complementary statistic.
For any three integers $r \leq k \leq n$, it is not hard to see that the statistic $\inv$
on $\OP_{n,k}^{(r)}$ achieves its maximum value at the unique point
$\sigma_0 := (k, k+1 \dots, n-1, n \mid k-1 \mid \cdots \mid 1) \in \OP_{n,k}^{(r)}$, and that 
\begin{equation}
\inv(\sigma_0) = (n-k)(k-1) + {k \choose 2}.
\end{equation}
We define the statistic $\coinv$ on $\OP_{n,k}^{(r)}$ by the rule 
\begin{equation}
\coinv(\sigma) := (n-k)(k-1) + {k \choose 2} - \inv(\sigma).
\end{equation}
For example, we have
\begin{equation*}
\coinv(25 \mid 1 \mid 34) = (5-3)(3-1) + {3 \choose 2} - \inv(25 \mid 1 \mid 34) = 
4 + 3 - 3 = 4.
\end{equation*}

It will be convenient to break up the coinversion statistic $\coinv$ into a sequence 
of smaller statistics.  Given an ordered set partition 
$\sigma = (B_1 \mid \cdots \mid B_k) \in \OP_{n,k}^{(r)}$, define the {\em coinversion code}
$\code(\sigma) = (c_1, c_2, \dots, c_n)$ as follows.  
Suppose $1 \leq i \leq n$ and $i \in B_j$.  Then
\begin{equation}
c_i = \begin{cases}
| \{ \ell > j \,:\, \min(B_{\ell}) > i \} | & \text{if $i = \min(B_j)$} \\
| \{ \ell > j \,:\, \min(B_{\ell}) > i \} | + (j-1) & \text{if $i \neq \min(B_j)$.}
\end{cases}
\end{equation}
The coinversion code of $(25 \mid 1 \mid 34)$ is therefore 
$\code(\sigma) = (c_1, c_2, c_3, c_4, c_5) = (1, 1, 0, 2, 0)$.  The coinversion code breaks
the statistic $\coinv$ into pieces.

\begin{proposition}
\label{coinv-decomposition}
Let $\sigma \in \OP_{n,k}^{(r)}$ with $\code(\sigma) = (c_1, c_2, \dots, c_n)$.
Then 
\begin{equation}
\coinv(\sigma) = c_1 + c_2 + \cdots + c_n.
\end{equation}
\end{proposition}

Which sequences $(c_1, c_2, \dots, c_n)$ of nonnegative integers can arise as the coinversion
code of some element $\sigma \in \OP_{n,k}^{(r)}$?  When $r = k = n$, these are precisely
the sequences $(c_1, c_2, \dots, c_n)$ which are componentwise $\leq$ the
staircase $(n-1, n-2, \dots, 0)$ of length $n$.  To state the answer for general
$r \leq k \leq n$, we will need some definitions.

If $S = \{s_1 < s_2 < \cdots < s_m\}$ is any subset of $[n]$, the {\em skip composition}
$\gamma(S) = (\gamma(S)_1, \dots, \gamma(S)_n)$ is the sequence given by
\begin{equation}
\gamma(S)_i = \begin{cases}
i - j + 1 & \text{if $i = s_j \in S$} \\
0 & \text{if $i \notin S$.}
\end{cases}
\end{equation}
We also let $\gamma(S)^* = (\gamma(S)_n, \dots, \gamma(S)_1)$ be the reversal
of the skip composition.  As an example, if $n = 7$ and $S = \{2,3,6\}$ then
$\gamma(S) = (0,2,2,0,0,4,0)$ and $\gamma(S)^* = (0,4,0,0,2,2,0)$.

\begin{theorem}
\label{coinversion-characterization}
Let $r \leq k \leq n$.  The map $\sigma \mapsto \code(\sigma)$ gives a bijection
from $\OP_{n,k}^{(r)}$ to the family $(c_1, \dots, c_n)$ of nonnegative integer sequences
such that 
\begin{itemize}
\item for all $r+1 \leq i \leq n$ we have $c_i < k$,
\item for all $1 \leq i \leq r$ we have $c_i < k - i + 1$, and
\item for any subset $S \subset [n]$ with $|S| = n-k+1$, the componentwise
inequality $\gamma(S)^* \leq (c_1, \dots, c_n)$ {\em fails} to hold.
\end{itemize}
\end{theorem}

\begin{proof}
Let $\CCC_{n,k}^{(r)}$ be the family of length $n$ sequences 
of nonnegative integers which satisfy the three conditions in the statement 
of the theorem.  Let $\sigma \in \OP_{n,k}^{(r)}$ with $\code(\sigma) = (c_1, \dots, c_n)$.
We show that $(c_1, \dots, c_n) \in \CCC_{n,k}^{(r)}$, so that the function
$\code: \OP_{n,k}^{(r)} \rightarrow \CCC_{n,k}^{(r)}$ is well-defined.  This is verified as follows.
\begin{itemize}
\item  For any $1 \leq i \leq n$, the block $B$ of $\sigma$ containing $i$ cannot contribute
to $c_i$, whereas each block $\neq B$ can contribute at most $1$ to $c_i$.
Consequently, we have $c_i < k$.
\item  Since $\sigma$ is $r$-Stirling, the letters $1, 2, \dots, r$ are all minimal in their blocks.
In particular, if $1 \leq i \leq r$,  the blocks containing $1, 2, \dots, i-1$ cannot contribute to 
$c_i$, so that $c_i < k - i + 1$.
\item  Finally, let $S \subseteq [n]$ satisfy $|S| = n-k+1$.  We verify 
$\gamma(S)^* \not\leq (c_1, \dots, c_n)$.  Working towards a
contradiction, suppose $\gamma(S)^* \leq (c_1, \dots, c_n)$.

Write the reversal $T := \{ n-i+1 \,:\, i \in S\}$ of $S$ as $T = \{t_1 < \cdots < t_{n-k+1} \}$.  
Since $\sigma$ has $n$ letters
and $k$ blocks, {\em at least one element of $T$ must be minimal in its block of $\sigma$}.  
If $t_{n-k+1}$ is minimal
in its block of $\sigma$, then 
\begin{align}
c_{t_{n-k+1}} &= \left| \left\{ 
\ell > t_{n-k+1} \,:\, 
\begin{array}{c} \text{$\ell$ is minimal in its block and } \\ \text{occurs to the right of $t_{n-k+1}$ in $\sigma$} 
\end{array}
\right\} \right| \\
& \leq | \{t_{n-k+1} + 1, \dots, n-1, n\}| \\
& = n - t_{n-k+1}.
\end{align}
But the term of $\gamma(S)^*$ in position $t_{n-k+1}$ is $n - t_{n-k+1} + 1$.
We conclude that $t_{n-k+1}$ is not minimal in its block of $\sigma$.
If $t_{n-k}$ were minimal in its block of $\sigma$, then
\begin{align}
c_{t_{n-k}} &= \left| \left\{ \ell > t_{n-k} \,:\, 
\begin{array}{c}
\text{$\ell$ is minimal in its block and} \\ \text{occurs to the right of $t_{n-k}$ in $\sigma$} \end{array} 
\right\} \right| \\
& \leq | \{t_{n-k} + 1, \dots, n-1, n\} - \{t_{n-k+1} \}| \\
& = n - t_{n-k} - 1,
\end{align}
But the term of $\gamma(S)^*$ in position $t_{n-k}$ is $n - t_{n-k}$.
We conclude 
that $t_{n-k}$ is not minimal in its block of $\sigma$.  If $t_{n-k-1}$ were minimal in its block of $\sigma$, the same 
reasoning leads to the contradiction
$c_{t_{n-k-1}} < n - t_{n-k-1} - 1$, etc.  We see that 
{\em none of the elements in $T$
are minimal in their block of $\sigma$}, a contradiction.  
\end{itemize}

In order to show that $\code: \OP_{n,k}^{(r)} \rightarrow \CCC_{n,k}^{(r)}$ is a bijection, 
we construct its inverse.
As this inverse will be defined using an insertion procedure,
we denote it $\iota: \CCC_{n,k}^{(r)} \rightarrow \OP_{n,k}^{(r)}$.

Let 
$(B_1 \mid \cdots \mid B_k)$ be a sequence of
of $k$ possibly empty sets of positive integers. We define the {\em coinversion label}
of the sets $B_1, \dots, B_k$ by labeling the empty sets with $0, 1, \dots, j$ from right to left (where there are $j+1$
empty sets), and then labeling the nonempty sets with $j+1, j+2, \dots, k-1$ from left to right. 
An example of coinversion labels is as follows, displayed as superscripts:
\begin{equation*}
(  \varnothing^2 \mid 1 3^3 \mid \varnothing^1 \mid 2 5^4 \mid 4^5 \mid \varnothing^0).
\end{equation*}
By construction, each of the letters $0, 1, \dots, k-1$ appears exactly once as a coinversion label.

Let $(c_1, \dots, c_n) \in \CCC_{n,k}^{(r)}$.
Then $0 \leq c_i \leq k-1$ for $1 \leq i \leq n$.
We define $\iota(c_1, \dots, c_n) = (B_1 \mid \cdots \mid B_k)$ recursively by starting with the sequence 
$(\varnothing \mid \cdots \mid \varnothing)$ of $k$ copies of the empty set, 
and for $i = 1, 2, \dots, n$ 
inserting $i$ into the unique block with coinversion label $c_i$.
Here is an example of this procedure for $(n,k,r) = (9,4,3)$ and 
$(c_1, \dots, c_9) = (2,0,1,1,1,0,2,1,3)$:
\begin{center}
\begin{tabular}{c | c | c}
$i$ & $c_i$ & $\sigma$  \\ \hline
$1$ & $2$ & $(\varnothing^3 \mid \varnothing^2 \mid \varnothing^1 \mid \varnothing^0 )$ \\
$2$ & $0$ & $(\varnothing^2 \mid 1^3 \mid \varnothing^1 \mid \varnothing^0 )$  \\
$3$ & $1$ & $(\varnothing^1 \mid 1^2 \mid \varnothing^0 \mid 2^3 )$ \\
$4$ & $1$ & $(3^1 \mid 1^2 \mid \varnothing^0 \mid 2^3 )$ \\
$5$ & $1$ & $(34^1 \mid 1^2 \mid \varnothing^0 \mid 2^3 )$ \\
$6$ & $0$ & $(345^1 \mid 1^2 \mid \varnothing^0 \mid 2^3 )$ \\
$7$ & $2$ & $(345^0 \mid 1^1 \mid 6^2 \mid 2^3 )$ \\
$8$ & $1$ & $(345^0 \mid 18^1 \mid 67^2 \mid 2^3 )$ \\
$9$ & $3$ & $(345^0 \mid 18^1 \mid 67^2 \mid 29^3 )$ \\
\end{tabular}
\end{center}
We conclude  
$\iota(2,0,1,1,1,0,2,1,3) = (345 \mid 18 \mid 67 \mid 29)$.

We verify that $\iota$ is a well-defined function 
$\CCC^{(r)}_{n,k} \rightarrow \OP^{(r)}_{n,k}$.  Let 
$(c_1, \dots, c_n) \in \CCC_{n,k}^{(r)}$ and let 
$\iota(c_1,\dots,c_n) = (B_1 \mid \cdots \mid B_k) = \sigma$.
We must show that $1, 2, \dots, r$ lie in distinct blocks of $\sigma$ and that $\sigma$ does not have any 
empty blocks.

Suppose there exist $1 \leq i < j \leq r$ such that $i$ and $j$ 
belong to the same block of $\sigma$.
Choose the pair $(i,j)$ to be lexicographically minimal with this property and suppose 
$i, j \in B_{\ell}$.
Since the sequence $(B_1 \mid \cdots \mid B_k)$ consists of $j-1$ singletons and 
$k-j+1$ copies of the empty set when $j$ is inserted by $\iota$, the definition of $\iota$
and the fact that $j$ was added to a non-singleton block imply
$c_j  \geq k - j + 1$, which contradicts the assumption
$(c_1, \dots, c_n) \in \CCC_{n,k}^{(r)}$.
We conclude that $1, 2, \dots, r$ lie in different blocks of $\sigma$.

Now suppose that some of the blocks of $\sigma = (B_1 \mid \cdots \mid B_k)$ are empty.  This means that at least 
$n-k+1$ of the letters in $[n]$ are {\em not} minimal in their block of $\sigma$.  Let $S$ be the lexicographically 
{\em first} set of $n-k+1$ letters in $[n]$ which are not minimal in their blocks.  We will derive the contradiction $\gamma(S)^* \leq (c_1, \dots, c_n)$.

Indeed, write the reversal $T = \{n-i+1 \,:\, i \in S \}$ of $S$ as $T = \{t_1 < \cdots < t_{n-k+1} \}$.
Let $1 \leq i \leq n-k+1$.  By our choice of $S$, we know that the letters in the set difference
\begin{equation}
\{ t_i + 1, t_i + 2, \dots, n \} - \{ t_{i+1}, t_{i+2}, \dots, t_{n-k+1} \}
\end{equation}
are all minimal in their blocks of $\sigma$; 
this set has $(n - t_i) - (n-k+1 - i) = k - t_i + i - 1$ elements.
Consequently, since $\sigma$ contains at least one empty block, 
when the $\iota$ inserts $t_i$, there
are $ \geq k - t_i + i$ empty blocks. 
This forces $c_{t_i} \geq k - t_i + i + 1$.  Since $k - t_i + i + 1$ is the term of
$\gamma(S)^*$ in position $t_i$, we conclude $\gamma(S)^* \leq (c_1, \dots, c_n)$,
which contradicts the assumption that $(c_1, \dots, c_n) \in \CCC_{n,k}^{(r)}$.
Therefore, none of the blocks of $\sigma$ are empty and the function 
$\iota: \CCC_{n,k}^{(r)} \rightarrow \OP_{n,k}^{(r)}$ is well-defined.
We leave it for the reader to check that $\code$ and $\iota$ are mutually 
inverse.
\end{proof}

The $\code$ bijection of Theorem~\ref{coinversion-characterization} will have 
algebraic importance to the theory of Gr\"obner bases.  Recall that a total order $<$
on monomials in $\QQ[\xx_n]$ is called a {\em monomial order} if 
\begin{itemize}
\item  $1 \leq m$ for any monomial $m$, and
\item  if $m_1, m_2,$ and $m_3$ are monomials with $m_1 < m_2$, we have
$m_1 \cdot m_3 < m_2 \cdot m_3$.
\end{itemize}
In this paper, we will exclusively use the {\em negative lexicographical term order}
$\neglex$ defined by $x_1^{a_1} \cdots x_n^{a_n} < x_1^{b_1} \cdots x_n^{b_n}$ if and only if
there exists $1 \leq i \leq n$ such that $a_i < b_i$ and $a_{i+1} = b_{i+1}, \dots, a_n = b_n$.

If $<$ is any monomial order and $f \in \QQ[\xx_n]$ is nonzero, let $\initial_<(f)$ be the leading
term of $f$.  Furthermore, if $I \subseteq \QQ[\xx_n]$ is an ideal, the {\em initial ideal} is 
$\initial_<(I) := \langle \initial_<(f) \,:\, f \in I - \{ 0 \} \rangle$.  A finite subset 
$G = \{g_1, \dots, g_s\} \subset I$ is called a {\em Gr\"obner basis} if
$\initial_<(I) = \langle \initial_<(g_1), \dots, \initial_<(g_s) \rangle$.  If $G$ is a 
Gr\"obner basis for $I$, we necessarily have $I = \langle G \rangle$.   
Every ideal $I \subseteq \QQ[\xx_n]$ has a Gr\"obner basis (with respect to some fixed 
monomial order $<$).  

Let $I \subseteq \QQ[\xx_n]$ be an ideal and fix a monomial order $<$.
If $G = \{g_1, \dots, g_s\}$ is a Gr\"obner basis for $I$, the set of monomials
\begin{equation}
\{ m \,:\, \initial_<(f) \nmid m \text{ for all $f \in I - \{ 0 \}$} \} =
\{ m \,:\, \initial_<(g_i) \nmid m \text{ for $1 \leq i \leq s$} \}
\end{equation}
descends to a $\QQ$-vector space basis for $\QQ[\xx_n]/I$.  This is called the 
{\em standard  basis} of $\QQ[\xx_n]/I$.  After a monomial order is fixed,
any quotient $\QQ[\xx_n]/I$ has a unique standard basis.
The $\code$ map precisely describes the standard basis of $R_{n,k}^{(r)}$ in terms of 
ordered $r$-Stirling partitions.

\begin{theorem}
\label{coinversion-standard-basis}
Let $r \leq k \leq n$ and consider the set of monomials $\MMM_{n,k}^{(r)}$ given by
\begin{equation} \MMM_{n,k}^{(r)} =
\left \{
x_1^{c_1} x_2^{c_2} \cdots x_n^{c_n} \,:\,  
(c_1, c_2, \dots, c_n) = \code(\sigma) \text{ for some 
$\sigma \in \OP_{n,k}^{(r)}$}
\right \}.
\end{equation}

1. The set $\MMM_{n,k}^{(r)}$
is the standard  basis for the $\QQ$-vector space $R_{n,k}^{(r)}$ with respect
to the $\neglex$ monomial order.

2. The set $\MMM_{n,k}^{(r)}$
is a $\ZZ$-basis for the $\ZZ$-module $S_{n,k}^{(r)}$. 
\end{theorem}

\begin{proof} 1.
We begin by proving the inequality $\dim(R_{n,k}^{(r)}) \geq | \OP_{n,k}^{(r)} |$. 
Consider $k$ distinct rational numbers 
$\alpha_1, \dots, \alpha_k$ and let $Y_{n,k}^{(r)} \subset \QQ^n$ be the family of points
$(y_1, \dots, y_n)$ such that 
\begin{itemize}
\item $\{ y_1, \dots, y_n\} = \{\alpha_1, \dots, \alpha_k\}$, and
\item the coordinates $y_1, \dots, y_r$ are distinct.
\end{itemize}
It is evident that $Y_{n,k}^{(r)}$ carries an action of the symmetric group product
$S_r \times S_{n-r}$, and that this affords an identification of $Y_{n,k}^{(r)}$ with 
$\OP_{n,k}^{(r)}$.

Let $\II(Y_{n,k}^{(r)}) \subseteq \QQ[\xx_n]$ be the ideal of polynomials in $\QQ[\xx_n]$ which
vanish on $Y_{n,k}^{(r)}$.  We have
\begin{equation}
\QQ[\xx_n]/\II(Y_{n,k}^{(r)}) \cong \QQ[Y_{n,k}^{(r)}] \cong \QQ[\OP_{n,k}^{(r)}]
\end{equation}
as  $S_r \times S_{n-r}$-modules.  If $f \in \II(Y_{n,k}^{(r)})$ is nonzero, let
$\tau(f)$ denote the homogeneous component of $f$ of highest degree and set
\begin{equation}
\TT(Y_{n,k}^{(r)}) := \langle \tau(f) \,:\, f \in \II(Y_{n,k}^{(r)}) - \{ 0 \} \rangle.
\end{equation}
We have the further $S_r \times S_{n-r}$-module isomorphism
\begin{equation}
\QQ[\xx_n]/\TT(Y_{n,k}^{(r)}) \cong 
\QQ[\xx_n]/\II(Y_{n,k}^{(r)}) \cong \QQ[Y_{n,k}^{(r)}] \cong \QQ[\OP_{n,k}^{(r)}].
\end{equation}
Proving the dimension inequality $\dim(R_{n,k}^{(r)}) \geq |\OP_{n,k}^{(r)}|$ therefore reduces
to showing the containment $I_{n,k}^{(r)} \subseteq \TT(Y_{n,k}^{(r)})$; we do this by considering 
the generators of $I_{n,k}^{(r)}$.
\begin{itemize}
\item  Let $1 \leq i \leq n$; we show that the monomial $x_i^k$ lies in $\TT(Y_{n,k}^{(r)})$.
This follows from the fact that 
$(x_i - \alpha_1) (x_i - \alpha_2) \cdots (x_i - \alpha_k) \in \II(Y_{n,k}^{(r)})$.
\item  We show that $e_n(\xx_n), e_{n-1}(\xx_n), \dots, e_{n-k+1}(\xx_n) \in \TT(Y_{n,k}^{(r)})$.  
Indeed, introduce
a new variable $t$ and consider the rational function
\begin{equation}
\frac{(1 - x_1 t) \cdots (1 - x_n t)}{(1 - \alpha_1 t) \cdots (1 - \alpha_k t)} =
\sum_{i,j} (-1)^i e_i(\xx_n) h_j(\alpha_1, \dots, \alpha_k) \cdot t^{i+j}.
\end{equation}
If $(x_1, \dots, x_n) \in Y_{n,k}^{(r)}$ the factors of the denominator cancel with $k$ factors
in the numerator, yielding a polynomial in $t$ of degree $n-k$.  If $n-k+1 \leq i \leq n$, taking
the coefficient of $t^i$ on both sides leads to $e_i(\xx_n) \in \TT(Y_{n,k}^{(r)})$.
\item
\begin{equation}
\frac{(1 - \alpha_1 t) \cdots (1 - \alpha_k t)}{(1 - x_1 t) \cdots (1 - x_r t)} =
\sum_{i,j} (-1)^i e_i(\alpha_1, \dots, \alpha_k) h_j(\xx_r) \cdot t^{i+j}.
\end{equation}
If $(x_1, \dots, x_n) \in Y_{n,k}^{(r)}$, the factors in the denominator cancel with $r$
factors in the numerator, yielding a polynomial in $t$ of degree $k-r$.  If 
$k-r+1 \leq j \leq k$, taking the coefficient of $t^i$ on both sides leads to
$h_j(\xx_r) \in \TT(Y_{n,k}^{(r)})$.
\end{itemize}
This completes the proof that $\dim(R_{n,k}^{(r)}) \geq |\OP_{n,k}^{(r)}|$.

Given any subset $S \subseteq [n]$ with reverse skip composition 
$\gamma(S)^* = (a_1, \dots, a_n)$, let $\xx(S)^* := x_1^{a_1} \cdots x_n^{a_n}$ be the 
associated {\em reverse skip monomial}.  By \cite[Sec. 3]{HRS}, we have
$\xx(S)^* \in \initial_<(I_{n,k}^{(r)})$ whenever $S \subseteq [n]$ satisfies $|S| = n-k+1$.
Furthermore, the identities
\begin{equation}
\label{homogeneous-identity}
h_d(x_1, \dots, x_{i-1}, x_i) - x_i h_{d-1}(x_1, \dots, x_{i-1}, x_i) =
h_d(x_1, \dots, x_{i-1})
\end{equation}
imply that $x_1^k, x_2^{k-1}, \dots, x_r^{k-r-1} \in \initial_<(I_{n,k}^{(r)})$.  Finally, we have
$x_{r+1}^k, \dots, x_{n-1}^k, x_n^k \in \initial_<(I_{n,k}^{(r)})$.
Theorem~\ref{coinversion-characterization} implies that the monomials in $\MMM_{n,k}^{(r)}$
are precisely those monomials in $\QQ[\xx_n]$ which are not divisible by any of the three 
classes of elements of $\initial_<(I_{n,k}^{(r)})$ listed above.   Again by
Theorem~\ref{coinversion-characterization} we have 
$\dim(R_{n,k}^{(r)}) \geq |\OP_{n,k}^{(r)}| = |\MMM_{n,k}^{(r)}|$, so that 
$\MMM_{n,k}^{(r)}$ is the standard  basis of $R_{n,k}^{(r)}$.  

2. From Item 1 of this theorem, we know that the set $\MMM_{n,k}^{(r)}$ descends to a linearly
independent subset of $S_{n,k}^{(r)}$; we need only show that $\MMM_{n,k}^{(r)}$ descends
to a $\ZZ$-spanning set of $S_{n,k}^{(r)}$.  To this end, let $m$ be any monomial in
$\ZZ[\xx_n]$.  We show inductively
that $m + J_{n,k}^{(r)}$ lies in the $\ZZ$-span of $\MMM_{n,k}^{(r)}$.  If 
$m \in \MMM_{n,k}^{(r)}$ this is obvious.  Otherwise, one of the following three things
must be true:
\begin{enumerate}
\item  There exists $1 \leq i \leq r$ such that $x_i^{k-i+1} \mid m$.
\item  There exists $r+1 \leq i \leq n$ such that $x_i^k \mid m$.
\item  There exists $S \subseteq [n]$ with $|S| = n-k+1$ such that $\xx(S)^* \mid m$.
\end{enumerate}

If (1) holds, Equation~\eqref{homogeneous-identity} implies  
$h_{k-i+1}(x_1, x_2, \dots, x_i) \in J_{n,k}^{(r)}$.  As a consequence, we have
\begin{equation}
x_i^{k-i+1} \equiv \text{a $\ZZ$-linear combination of monomials $< x_i^{k-i+1}$ in
$\neglex$} \, \, \,
\text{(mod $J_{n,k}^{(r)}$)}.
\end{equation}
If we multiply through by the monomial $m/x_i^{k-i+1}$, we see that
\begin{equation}
m \equiv \text{a $\ZZ$-linear combination of monomials $< m$ in
$\neglex$} \, \, \,
\text{(mod $J_{n,k}^{(r)}$)},
\end{equation}
so that inductively we see that $m + J_{n,k}^{(r)}$ lies in the span of $\MMM_{n,k}^{(r)}$.

If (2) holds, then $m \in J_{n,k}^{(r)}$, so certainly $m + J_{n,k}^{(r)} = 0$ lies in the
$\ZZ$-span of $\MMM_{n,k}^{(r)}$.

If (3) holds, let $\kappa_{\gamma(S)^*}(\xx_n) \in \ZZ[\xx_n]$ be the 
{\em Demazure character} attached to the reverse skip composition $\gamma(S)^*$. 
This is a certain polynomial in the variables $x_1, \dots, x_n$ with nonnegative integer
coefficients.  The precise form of this polynomial is not important for us, but we have
(see e.g. \cite[Lem. 3.5]{HRS})
\begin{equation}
\label{demazure-form}
\kappa_{\gamma(S)^*}(\xx_n) = \xx(S)^* +
\text{a $\ZZ$-linear combination of terms $< \xx(S)^*$ in $\neglex$}.
\end{equation}
By \cite[Lem 3.4]{HRS} we have $\kappa_{\gamma(S)^*}(\xx_n) \in J_{n,k}^{(r)}$, so that
Equation~\eqref{demazure-form} implies
\begin{equation}
\label{demazure-congruence}
\xx(S)^* \equiv \text{a $\ZZ$-linear combination of terms $< \xx(S)^*$ in $\neglex$} \, \, \,
\text{(mod $J_{n,k}^{(r)}$)}.
\end{equation}
If we multiply Equation~\eqref{demazure-congruence} through by the monomial
$m/\xx(S)^*$, we get
\begin{equation}
m \equiv \text{a $\ZZ$-linear combination of terms $< m$ in $\neglex$} \, \, \,
\text{(mod $J_{n,k}^{(r)}$)}.
\end{equation}
so that inductively we see that $m + J_{n,k}^{(r)}$ lies in the $\ZZ$-span of 
$\MMM_{n,k}^{(r)}$.
\end{proof}

When $r = 0$, Theorem~\ref{coinversion-standard-basis} is equivalent to a result of 
Haglund, Rhoades, and Shimozono \cite[Thm. 4.13]{HRS}.  However, the proof
of Theorem~\ref{coinversion-standard-basis} is much more direct that 
of \cite[Thm. 4.13]{HRS} (and those in \cite[Sec. 4]{HRS} in general); whereas 
we associate an explicit standard basis element $x_1^{c_1} \cdots x_n^{c_n}$ to 
any ordered set partition $\sigma$, the description of the standard bases in \cite{HRS}
is recursive in nature.  We exhibit this link between ordered set partitions and 
standard basis elements with an example.

\begin{example}
To illustrate Theorem~\ref{coinversion-standard-basis},
we give the standard  basis of $R_{4,3}^{(2)}$ with respect to $\neglex$.
\begin{footnotesize}
\begin{center}
\begin{tabular}
{c | c | c}
$\sigma$ & $\code(\sigma)$ & $\mathrm{monomial} $  \\  \hline 
$(1 \mid 2 \mid 34)$ & $(2, 1, 0, 2)$ & $x_1^2 x_2 x_4^2$ \\
$(1 \mid 24 \mid 3)$ & $(2,1,0,1)$ & $x_1^2 x_2 x_4$ \\
$(14 \mid 2 \mid 3)$ & $(2,1,0,0)$ & $x_1^2 x_2$ \\
$(1 \mid 23 \mid 4)$ & $(2,1,2,0)$ & $x_1^2 x_2 x_3^2$ \\
$(13 \mid 2 \mid 4)$ & $(2,1,1,0)$ & $x_1^2 x_2 x_3$ \\
$(2 \mid 1 \mid 34)$ & $(1,1,0,2)$ & $x_1 x_2 x_4^2$ \\
$(2 \mid 14 \mid 3)$ & $(1,1,0,1)$ & $x_1 x_2 x_4$ \\
$(24 \mid 1 \mid 3)$ & $(1,1,0,0)$ & $x_1 x_2$ \\
$(2 \mid 13 \mid 4)$ & $(1,1,2,0)$ & $x_1 x_2 x_3^2$ \\
$(23 \mid 1 \mid 4)$ & $(1,1,1,0)$ & $x_1 x_2 x_3$ \\
\end{tabular} 
\quad
\begin{tabular}
{c | c | c}
$\sigma$ & $\code(\sigma)$ & $\mathrm{monomial}$   \\  \hline 
$(1 \mid 34 \mid 2)$ & $(2,0,0,1)$ & $x_1^2 x_4$ \\
$(1 \mid 3 \mid 24)$ & $(2,0,0,2)$ & $x_1^2 x_4^2$ \\
$(14 \mid 3 \mid 2)$ & $(2,0,0,0)$ & $x_1^2$ \\
$(1 \mid 4 \mid 23)$ & $(2,0,2,0)$ & $x_1^2 x_3^2$ \\
$(13 \mid 4 \mid 2)$ & $(2,0,1,0)$ & $x_1^2 x_3$ \\
$(2 \mid 34 \mid 1)$ & $(0,1,0,1)$ & $x_2 x_4$ \\
$(2 \mid 3 \mid 14)$ & $(0,1,0,2)$ & $x_2 x_4^2$ \\
$(24 \mid 3 \mid 1)$ & $(0,1,0,0)$ & $x_2$ \\
$(2 \mid 4 \mid 13)$ & $(0,1,2,0)$ & $x_2 x_3^2$ \\
$(23 \mid 4 \mid 1)$ & $(0,1,1,0)$ & $x_2 x_3$ \\
\end{tabular} 
\quad
\begin{tabular}
{c | c | c}
$\sigma$ & $\code(\sigma)$ & $\mathrm{monomial}$  \\  \hline 
$(34 \mid 1 \mid 2)$ & $(1,0,0,0)$ & $x_1$ \\
$(3 \mid 14 \mid 2)$ & $(1,0,0,1)$ & $x_1 x_4$ \\
$(3 \mid 1 \mid 24)$ & $(1,0,0,2)$ & $x_1 x_4^2$ \\
$(4 \mid 13 \mid 2)$ & $(1,0,1,0)$ & $x_1 x_3$ \\
$(4 \mid 1 \mid 23)$ & $(1,0,2,0)$ & $x_1 x_3^2$ \\
$(34 \mid 2 \mid 1)$ & $(0,0,0,0)$ & $1$ \\
$(3 \mid 24 \mid 1)$ & $(0,0,0,1)$ & $x_4$ \\
$(3 \mid 2 \mid 14)$ & $(0,0,0,2)$ & $x_4^2$ \\
$(4 \mid 23 \mid 1)$ & $(0,0,1,0)$ & $x_3$ \\
$(4 \mid 2 \mid 13)$ & $(0,0,2,0)$ & $x_3^2$ \\
\end{tabular} 
\end{center}
\end{footnotesize}
\end{example}

As an application of Theorem~\ref{coinversion-standard-basis}, we can describe the Hilbert
series of $R_{n,k}^{(r)}$ in terms of the $\coinv$ statistic.

\begin{corollary}
\label{hilbert}
The Hilbert series of $R_{n,k}^{(r)}$ is given by
\begin{equation}
\Hilb(R_{n,k}^{(r)};q) =
\sum_{\sigma \in \OP_{n,k}^{(r)}} q^{\coinv(\sigma)}.
\end{equation}
\end{corollary}

As another application of Theorem~\ref{coinversion-standard-basis}, we can
describe the ungraded isomorphism type of $R_{n,k}^{(r)}$ as a module over
$S_r \times S_{n-r}$.
When $r = k = n$, this is 
Chevalley's classical result \cite{C} 
that the coinvariant ring is isomorphic to the regular representation
of $S_n$.

\begin{corollary}
\label{chevalley}
We have an isomorphism of ungraded $S_r \times S_{n-r}$-modules
\begin{equation}
R_{n,k}^{(r)} \cong \QQ[\OP_{n,k}^{(r)}].
\end{equation}
\end{corollary}

It seems that the isomorphism type of $R_{n,k}^{(r)}$ as a {\em graded}
$S_r \times S_{n-r}$-module can be described in terms of known
graded modules by the (graded) tensor product decomposition
\begin{equation}
\label{tensor-conjecture}
R_{n,k}^{(r)} \cong R_r \otimes_{\CC} \varepsilon_r R_{n,k}.
\end{equation}
In the conjectural isomorphism \eqref{tensor-conjecture} of graded $S_r \times S_{n-r}$-modules,
\begin{itemize}
\item  $R_r = \QQ[\xx_r]/\langle e_1(\xx_r), \dots, e_r(\xx_r)\rangle$ is the classical coinvariant
ring in the first $r$ variables $\xx_r$, with its graded action of $S_r$,
\item  $R_{n,k} = R_{n,k}^{(0)}$ is the graded $S_n$-module
$\QQ[\xx_n]/\langle x_1^k, \dots, x_n^k, e_n(\xx_n), \dots, e_{n-k+1}(\xx_n) \rangle$,  and
\item  $\varepsilon_r \in \QQ[S_n]$ is the group algebra element
\begin{equation}
\varepsilon_r := \sum_{w \in S_r} \mathrm{sign}(w) \cdot w 
\end{equation}
which antisymmetrizes
over the subgroup $S_r \subseteq S_n$ (acting on the first $r$ letters), so that $S_{n-r}$ 
(acting on the last $n-r$ letters) commutes with $\varepsilon_r$ and therefore
\item  $\varepsilon_r R_{n,k}$ is naturally a $S_{n-r}$-module, and 
\item the action of the product group $S_r \times S_{n-r}$ on the tensor product is given by
\begin{equation}
(w_1 \times w_2).(v_1 \otimes v_2) := (w_1.v_1) \otimes (w_2.v_2).
\end{equation}
\end{itemize}

\section{Line configurations and $r$-Stirling partitions}
\label{Line}

We shift focus from algebra to geometry and initiate the study of
$X_{n,k}^{(r)}$.
In order to study the variety $X_{n,k}^{(r)}$, we will need to break it into pieces 
in a reasonable way.  For this we will use the notion of an {\em affine paving} (called
a {\em cellular decomposition} in \cite{PR}).

 Let $X$ be a smooth irreducible complex algebraic variety.
An {\em affine paving} of $X$ is an ordered 
partition 
\begin{equation}
X = C_1 \sqcup \cdots \sqcup C_m 
\end{equation}
such that 
\begin{itemize}
\item  for all $i$, the union $C_1 \sqcup \cdots \sqcup C_i$ is a closed subvariety
of $X$, and
\item  $C_i$ is isomorphic as a variety to the affine space $\CC^{n_i}$, for some integer $n_i$.
\end{itemize}
The $C_i$ are referred to as the {\em cells} of the affine paving and we will say
that the partition $\{C_1, \dots, C_m \}$ {\em induces} an affine paving of $X$.
In this situation, the classes of the cell closures 
$\{ [\overline{C_1}], \dots, [\overline{C_m}] \}$ give a $\ZZ$-basis
for the (singular) cohomology ring $H^{\bullet}(X)$.

The projective space $\PP^{k-1}$ has an affine paving induced by the cells
$\{C_1, C_2, \dots, C_k\}$, where 
\begin{equation}
C_i = \{ [ x_1 : x_2 : \cdots : x_k ] \in \PP^{k-1} \,:\, 
x_1 = \cdots = x_{i-1} = 0 \text{ and } x_i \neq 0 \}.
\end{equation}
Taking products of these cells gives 
the standard affine paving of $(\PP^{k-1})^n$ whose cells are indexed
by words $w = w_1 \dots w_n \in [k]^n$.  Following \cite{PR}, we will consider 
a {\em different} affine paving of $(\PP^{k-1})^n$ whose cells are again indexed
by words in $[k]^n$.  In order to describe this paving, we will need some terminology.

Let $\Mat_{k \times n}$ stand for the affine space of all complex $k \times n$ matrices
$m$.
Let  $\UUU_{n,k}^{(r)}$ be the Zariski open subset
\begin{equation}
\UUU_{n,k}^{(r)} := \left\{
m \in \Mat_{k \times n} \,:\,
\begin{array}{c}
\text{the matrix $m$ has full rank, no zero} \\ \text{columns, and the first $r$ columns} \\
\text{of $m$ are linearly independent}
\end{array}
\right\}.
\end{equation}
If we let $T \subset GL_n$ be the rank $n$ diagonal torus, then $T$ acts freely on the columns
of $\UUU_{n,k}^{(r)}$ and we may identify the orbit space as 
$\UUU_{n,k}^{(r)}/T = X_{n,k}^{(r)}$.  Furthermore, we consider the larger Zariski open
set $\VVV_{n,k}$ given by
\begin{equation}
\VVV_{n,k} := \{ m \in \Mat_{k \times n} \,:\, \text{$m$ has no zero columns} \}.
\end{equation}
This time we have the identification $\VVV_{n,k}/T = (\PP^{k-1})^n$.

Let $w = w_1 \dots w_n \in [k]^n$ be a word in the letters $1, 2, \dots, k$ of length $n$.
An index $1 \leq j \leq n$ is called {\em initial} if $w_j$ is the first occurrence of its letter in $w$;
let $\initial(w) = \{j_1 < j_2 < \cdots < j_s\}$ be the set of initial indices in $w$.  For example,
if $w = 242141 \in [4]^6$ then $\initial(w) = \{1, 2, 4\}$.
The $k \times n$ {\em pattern matrix} $\PM(w)$ has entries in the set $\{0, 1, \star\}$ as follows:
\begin{equation}
\PM(w)_{i,j} = \begin{cases}
1 & \text{if $w_j = i$} \\
0 & \text{if the letter $i$ does not appear in $w$} \\
\star & \text{if $j \in \initial(w), i < w_j$, and there exists $j' < j$ such that $w_{j'} = i$} \\
0 & \text{if $j \in \initial(w)$ and ($i > w_j$ or there does not exist $j' < j$ such that 
$w_{j'} = i$)} \\
\star & \text{if $j \notin \initial(w), i \neq w_j,$ and the first $i$ appears before the 
first $w_j$ in $w$} \\
0 & \text{if $j \notin \initial(w), i \neq w_j,$ and the first $i$ appears after the first $w_j$ in $w$.}
\end{cases}
\end{equation}
In our example,
\begin{equation*}
\PM(242141) = \begin{pmatrix}
0 & 0 & 0 & 1 & 0 & 1 \\
1 & \star & 1 & 0 & \star & \star \\
0 & 0 & 0 & 0 & 0 & 0 \\
0 & 1 & 0 & 0 & 1 & \star
\end{pmatrix}.
\end{equation*}

For any word $w = w_1 \dots w_n \in [k]^n$,
let $\widehat{C_w}$ be the affine space of all matrices obtained by replacing the $\star$'s in
$\PM(w)$ by complex numbers. 
Let $U \subset GL_k(\CC)$ be the unipotent subgroup of {\em lower} triangular matrices
with $1$'s on the diagonal.
We define a subset $C_w \subseteq (\PP^{k-1})^n$ by
\begin{equation}
C_w := \text{image of $U \cdot \widehat{C_w}$ in $(\PP^{k-1})^n$.}
\end{equation}
It follows from \cite{PR} that $C_w$ is isomorphic as a variety to an affine space.

\begin{proposition}
\label{projective-paving}
(\cite{PR})
For any $k \leq n$, the set $\{ C_w \,:\, w \in [k]^n \}$ induces an affine paving of 
$(\PP^{k-1})^n$.
\end{proposition}

The affine paving of Proposition~\ref{projective-paving} induces an affine paving of 
$X_{n,k}^{(r)}$.  To describe this paving, we define $\WWW_{n,k}^{(r)}$ to be the family
of words $w= w_1 w_2 \dots w_n \in [k]^n$ such that the letters $1, 2, \dots, k$ all appear in
$w$ and that the first $r$ letters $w_1, w_2, \dots, w_r$ of $w$ are distinct.

\begin{proposition}
\label{affine-paving}
The family of cells $\{ C_w \,:\, w \in \WWW_{n,k}^{(r)} \}$ induces an affine paving 
of the variety $X_{n,k}^{(r)}$.
\end{proposition}

\begin{proof}
Let $w \in [k]^n$ be any word and consider the cell $C_w \subset (\PP^{k-1})^n$.
The definition of the pattern matrix $\PM(w)$ implies that $C_w \subset X_{n,k}^{(r)}$ if
$w \in \WWW_{n,k}^{(r)}$ and $C_w \cap X_{n,k}^{(r)} = \varnothing$ otherwise.
Now observe that the total order on the cells 
$\{ C_w \,:\, w \in [k]^n \}$ inducing the affine paving of Proposition~\ref{projective-paving}
may be taken to start with those $w \notin \WWW_{n,k}^{(r)}$ (in some order)
and end with those $w \in \WWW_{n,k}^{(r)}$ (in some order).  The claim
follows.
\end{proof}

Our next task is to present the cohomology of $X_{n,k}^{(r)}$ as the quotient
$S_{n,k}^{(r)}$ and describe the images of the $\ZZ$-basis 
$\{ [\overline{C_w}] \,:\, w \in \WWW_{n,k}^{(r)} \}$ afforded by Proposition~\ref{affine-paving}.
We being by recalling the standard presentation of the cohomology of $(\PP^{k-1})^n$.

The cohomology of $(\PP^{k-1})^n$ is presented as
\begin{equation}
\label{projective-presentation}
H^{\bullet}((\PP^{k-1})^n) = \ZZ[\xx_n]/ \langle x_1^k, \dots, x_n^k \rangle,
\end{equation}
where $x_i$ represents the Chern class $c_1(\ell_i^*) \in H^2( (\PP^{k-1})^n)$ of the 
dual to the $i^{th}$ tautological line bundle $\ell_i \twoheadrightarrow (\PP^{k-1})^n$.

Given a word $w \in [k]^n$, a polynomial representative for 
$[\overline{C_w}] \in H^{\bullet}((\PP^{k-1})^n)$ was calculated in \cite{PR}.
In order to state it, we recall the classical {\em Schubert polynomials}
attached to permutations in $S_n$.

The {\em Schubert polynomials} $\{ \symm_w \,:\, w \in S_n \}$ are defined recursively
by
\begin{equation}
\begin{cases}
\symm_{w_0} = x_1^{n-1} x_2^{n-2} \cdots x_n^0 & \text{for $w_0 = n(n-1) \dots 1$} \\
\symm_{w s_i} = \partial_i \symm_w & \text{if $w_i > w_{i+1}$}.
\end{cases}
\end{equation}
Here $w s_i$ is the permutation whose one-line notation 
$w s_i = w_1 \dots w_{i+1} w_i \dots w_n$ is obtained from that of $w$ by interchanging
the letters in positions $i$ and $i+1$ and $\partial_i$ is the {\em divided difference operator}
\begin{equation}
\partial_i(f(x_1, \dots, x_n)) = 
\frac{f(x_1, \dots, x_i, x_{i+1}, \dots, x_n) - f(x_1, \dots, x_{i+1}, x_i, \dots, x_n)}{x_i - x_{i+1}}.
\end{equation}

In order to extend Schubert polynomials from permutations in $S_n$ to words in $[k]^n$, we 
will need some notation.  A word $w$ is called {\em convex} if it does not have a subword 
of the form $\dots i \dots j \dots i \dots$.  Any word $w$ has a unique 
convexification $\conv(w)$ which is characterized by being convex, having the same letter
multiplicities as $w$, and having its initial letters appear in the same order from left to right.
For example, we have $\conv(242141) = 224411$.
Furthermore, let $\sigma(w) \in S_n$ be the unique permutation with a minimal number of 
inversions which sorts $w$ to $\conv(w)$; in our example
$\sigma(242141) = 132546 \in S_6$.

Suppose $w = w_1 \dots w_n \in [k]^n$ is a convex word with $m$ distinct letters.
Let $\{ i_1 < i_2 < \cdots < i_{k-m} \}$ be the letters in $[k]$ which do {\em not} appear in $w$.
We define the {\em standardization} $\st(w) = \st(w)_1 \dots \st(w)_{n+k-m} \in S_{n+k-m}$
to be the permutation obtained from $w$ by fixing the initial letters of $w$, replacing the 
non-initial letters of $w$ from left to right with $k+1, k+2, \dots, n+k-m$, and 
appending the sequence $i_1 i_2 \dots i_{k-m}$ to the end.  For example, if 
$(n,k) = (7,5)$ and $w = 3344411$ then 
$\st(w) = 364781925 \in S_9$.

Let $w \in [k]^n$ be an arbitrary word of length $n$ in the letters $1, 2, \dots, k$.  
The {\em word Schubert polynomial} $\symm_w$
is defined by
\begin{equation}
\symm_w := \sigma(w)^{-1}.\symm_{\st(\conv(w))}.
\end{equation}
Although the permutation $\st(\conv(w))$ will lie in a symmetric group of rank $> n$ 
when $w$ does not contain all of the letters $1, 2, \dots, k$, the polynomial
$\symm_w$ depends only on the variables $x_1, x_2, \dots, x_n$ so that 
$\symm_w \in \ZZ[\xx_n]$.  Pawlowski and Rhoades proved \cite{PR}
that the closure of the cell $C_w$ is represented by $\symm_w$ 
under the presentation \eqref{projective-presentation}:
\begin{equation}
\label{word-representative}
[\overline{C_w}] \text{ is represented by } \symm_w \text{ in } H^{\bullet}((\PP^{k-1})^n).
\end{equation}

\begin{theorem}
\label{cohomology-presentation}
Let $r \leq k \leq n$. The singular cohomology of $X_{n,k}^{(r)}$ may be presented as
\begin{equation}
\label{presentation-equation}
H^{\bullet}(X_{n,k}^{(r)}) = S_{n,k}^{(r)}.
\end{equation}
Furthermore, under the 
presentation \eqref{presentation-equation}, if $w \in \WWW_{n,k}^{(r)}$ the cell
closure $\overline{C_w}$ is represented in $H^{\bullet}(X_{n,k}^{(r)})$ by $\symm_w$.
\end{theorem}

\begin{proof}
Consider the affine paving $\{ C_w \,:\, w \in [k]^n \}$ of $(\PP^{k-1})^n$ afforded
by Proposition~\ref{projective-paving}.  If $w \notin \WWW_{n,k}^{(r)}$, we have
$\overline{C_w} \cap X_{n,k}^{(r)} = \varnothing$.  By Proposition~\ref{affine-paving},
it follows that $X_{n,k}^{(r)}$ is obtained from $(\PP^{k-1})^n$ by excising 
the union of cell closures $\bigcup_{w \in [k]^n - \WWW_{n,k}^{(r)}} \overline{C_w}$.
It follows (see \cite{PR}) that the cohomology ring $H^{\bullet}(X_{n,k}^{(r)})$ may be 
presented as 
\begin{equation}
\label{quotient-presentation}
H^{\bullet}(X_{n,k}^{(r)}) = H^{\bullet}((\PP^{k-1})^n)/J,
\end{equation}
where $J \subseteq H^{\bullet}((\PP^{k-1})^n)$ is the ideal generated by those 
$[\overline{C_w}]$ for which $w \in [k]^n - \WWW_{n,k}^{(r)}$.
If we use the presentation of $H^{\bullet}((\PP^{k-1})^n)$ given in
\eqref{projective-presentation} together with 
the polynomial representatives \eqref{word-representative},
we can write
\begin{equation}
\label{quotient-presentation-two}
H^{\bullet}(X_{n,k}^{(r)}) = \ZZ[\xx_n]/I,
\end{equation}
where $I \subseteq \ZZ[\xx_n]$ is the ideal generated by 
$x_1^k, x_2^k, \dots, x_n^k$ together with $\{ \symm_w \,:\, w \in [k]^n - \WWW_{n,k}^{(r)} \}$.

{\bf Claim:} {\em We have $J_{n,k}^{(r)} \subseteq I$.}  

To prove the Claim, we show that every generator of $J_{n,k}^{(r)}$ lies in $I$. We handle
each type of generator separately.
\begin{itemize}
\item
The generators $x_1^k, x_2^k, \dots, x_n^k$ of $J_{n,k}^{(r)}$ are also generators of $I$.
\item
For the generators $e_{n-i+1}(\xx_n)$ (where $1 \leq i \leq k$) of $J_{n,k}^{(r)}$ we do 
the following.  For $1 \leq i \leq k$ let $w^i$ be the unique weakly increasing word
in $[k]^n$ containing exactly the letters $[k] - \{i\}$ and whose first $k-1$ letters are 
distinct.  For example, the word $w^3 \in [5]^7$ is  $w^3 = 1245555$.
Since $i$ does
not appear in $w^i$, we have $w^i \notin \WWW_{n,k}^{(r)}$, 
so that $\symm_{w^i}$ is a generator
of $I$.
Furthermore, we have 
\begin{equation*}
\st(\conv(w^i)) = 12 \dots (i-1)(i+1) \dots n(n+1)i \in S_{n+1}
\end{equation*}
which 
implies $\symm_{w^i} = e_{n-i+1}(\xx_n)$.
\item  Finally, we consider the generators $h_{k-i+1}(\xx_r)$ (where $1 \leq i \leq r$)
of $J_{n,k}^{(r)}$.  These generators are not in general generators of $I$, but we show 
that they nevertheless are contained in $I$.  
If $k = n$ then $X_{n,k}^{(r)} = X_{n,n}$ so that the theorem follows from \cite{PR};
we assume that $k < n$.

For $1 \leq i \leq r-1$, let $v^i \in [k]^n$
be the following weakly increasing word:
\begin{equation*}
v^i = 12 \dots (i-1)ii(i+1)(i+2) \dots (k-1)k \dots k.
\end{equation*}
For example, the word  $v^3 \in [5]^7$ is $v^3 = 12334555$.  Since $k < n$, every letter
in $[k]$ appears in $v^i$.  However, since the first $r$ letters of $v^i$
are not distinct, we have $v^i \notin \WWW_{n,k}^{(r)}$, so that 
$\symm_{v^i}$ is a generator of $I$.  We have
\begin{equation*}
\st(\conv(v^i)) = 1 2 \dots (i-1)i(k+1)(i+1)(i+2) \dots n \in S_n
\end{equation*}
which implies $\symm_{v^i} = h_{k-i}(\xx_{i+1})$.

The above paragraph shows that 
\begin{equation*}
h_{k-r+1}(\xx_r), h_{k-r+2}(\xx_{r-1}), \dots, h_{k-1}(\xx_2) \in I.  
\end{equation*}
The variable power
$h_k(\xx_1) = x_1^k$ also lies in $I$.  The identity
\begin{equation}
h_d(x_1, \dots, x_{i-1}, x_i) = x_i \cdot h_{d-1}(x_1, \dots, x_{i-1}, x_i) +
h_d(x_1, \dots, x_{i-1})
\end{equation}
together with the fact that $I$ is an ideal in $\ZZ[\xx_n]$ can be used to show that 
\begin{equation*}
h_{k-r+1}(\xx_r), h_{k-r+2}(\xx_r), \dots, h_k(\xx_r) \in I, 
\end{equation*}
which is what we wanted to show.
This completes the proof of the Claim.
\end{itemize}

By our Claim, we have a canonical surjection of $\ZZ$-modules
\begin{equation}
S_{n,k}^{(r)} = \ZZ[\xx_n]/J_{n,k}^{(r)} \twoheadrightarrow 
\ZZ[\xx_n]/I = H^{\bullet}(X_{n,k}^{(r)}).
\end{equation}
By Theorem~\ref{coinversion-standard-basis}, the module 
$S^{(r)}_{n,k}$ is a free $\ZZ$-module of rank 
$|\OP_{n,k}^{(r)}|$.  
By Proposition~\ref{affine-paving}, the cohomology ring 
$H^{\bullet}(X_{n,k}^{(r)})$ is a free $\ZZ$-module of rank
$|\WWW_{n,k}^{(r)}|$.  Since we have $|\OP_{n,k}^{(r)}| = |\WWW_{n,k}^{(r)}|$ and 
any surjection between free $\ZZ$-modules of the same rank must be an isomorphism,
we obtain the presentation~\eqref{presentation-equation} of the cohomology
of $X_{n,k}^{(r)}$.  The last sentence of the theorem follows
from \eqref{word-representative}.
\end{proof}

The cohomology representatives of the cell closures in any affine paving of a smooth irreducible variety
$X$ give rise to a $\ZZ$-basis for the cohomology ring  $H^{\bullet}(X)$.
Theorem~\ref{cohomology-presentation} therefore yields the following immediate corollary.

\begin{corollary}
\label{schubert-basis}
Let $r \leq k \leq n$.  The set of polynomials 
$\{ \symm_w \,:\, w \in \WWW_{n,k}^{(r)} \}$ descends to a $\ZZ$-basis for 
$S_{n,k}^{(r)}$.
\end{corollary}

We have the following isomorphisms of ungraded
$S_r \times S_{n-r}$-modules:
\begin{equation}
\label{ungraded-geometric-isomorphism} 
H^{\bullet}(X_{n,k}^{(r)}; \QQ) \cong
\QQ \otimes_{\ZZ} H^{\bullet}(X_{n,k}^{(r)}) \cong \QQ \otimes_{\ZZ} S_{n,k}^{(r)} \cong
R^{(r)}_{n,k} \cong \QQ[\OP_{n,k}^{(r)}].
\end{equation}
The first of these isomorphisms follows from the Universal Coefficient Theorem (see e.g. \cite{Hatcher})
and the fact that 
$H^{\bullet}(X_{n,k}^{(r)})$ vanishes in odd degree.
The second is Theorem~\ref{cohomology-presentation}. The third follows from the definitions
of $S_{n,k}^{(r)}$ and $R_{n,k}^{(r)}$. The fourth follows from 
Corollary~\ref{chevalley}.
The space $X_{n,k}^{(r)}$ of line configurations therefore gives a geometric model for 
ordered $r$-Stirling partitions.
It may be possible to exploit this geometric model to describe the {\em graded} structure of 
$R_{n,k}^{(r)}$ as follows; the authors thank an anonymous referee for pointing this out.

Let $G(r,k)$ be the Grassmannian of $r$-dimensional subspaces  $V \subseteq \CC^k$ and consider the  
subspace $Y_{n,k}^{(r)} \subseteq G(r,k) \times (\PP^{k-1})^{n-r}$ defined as follows
\begin{equation}
Y_{n,k}^{(r)} := \{ (V, \ell_{r+1}, \dots, \ell_n) \,:\, V + \ell_{r+1} + \cdots + \ell_n = \CC^k \}.
\end{equation}
The space $Y_{n,k}^{(r)}$ is an open subvariety of $G(r,k) \times (\PP^{k-1})^{n-r}$.  We have a natural map
\begin{center}
\begin{tikzpicture}
\node(E) at (-2.2,0) {$\pi:$};

\node(A) at (0,0) {$X_{n,k}^{(r)}$};

\node(B) at (6,0) {$Y_{n,k}^{(r)}$};

\node(C) at (0,-1) {$(\ell_1, \dots, \ell_r, \ell_{r+1}, \dots, \ell_n)$};

\node(D) at (6,-1) {$(\ell_1 + \dots + \ell_r, \ell_{r+1}, \dots, \ell_n)$};

\draw[->] (A) -- (B);

\draw[|->] (C) -- (D);
\end{tikzpicture}
\end{center}
obtained by taking the (necessarily $r$-dimensional) span of the first $r$ lines in a typical configuration
in $X_{n,k}^{(r)}$.

The map $\pi: X_{n,k}^{(r)} \rightarrow Y_{n,k}^{(r)}$ is a fiber bundle. The fiber $F$ over a  point 
$(V, \ell_{r+1}, \dots, \ell_n) \in Y_{n,k}^{(r)}$ is given by the space of $r$-tuples 
$(\ell_1, \dots, \ell_r)$
of linearly independent
lines in the $r$-dimensional vector space $V$, which is homotopy equivalent to the flag variety 
$\mathcal{F \ell}(r)$. 
The inclusion $\iota: F \hookrightarrow X_{n,k}^{(r)}$ induces a map on rational cohomology
$\iota^*: H^{\bullet}(X_{n,k}^{(r)}; \QQ) \rightarrow H^{\bullet}(F; \QQ)$.
Since $H^{\bullet}(F;\QQ)$ is generated by the Chern 
classes $c_1(\ell_1^*), \dots, c_1(\ell_r^*)$
of the tautological line bundles 
$\ell_1^*, \dots, \ell_r^*$ over $F$, and these line bundles are pullbacks under $\iota$ of the corresponding
bundles on $X_{n,k}^{(r)}$, the map $\iota^*$ is a surjection.

By the last paragraph,
the Leray-Hirsch Theorem (see e.g. \cite{Hatcher})  provides the following isomorphism of 
$H^{\bullet}(Y_{n,k}^{(r)}; \QQ)$-modules:
\begin{equation}
\label{leray-hirsch} H^{\bullet}(X_{n,k}^{(r)}; \QQ) \cong
H^{\bullet}(F; \QQ) \otimes_{\QQ} H^{\bullet}(Y_{n,k}^{(r)}; \QQ) .
\end{equation}
The isomorphism \eqref{leray-hirsch} seems quite close to the 
conjectural isomorphism \eqref{tensor-conjecture}.
The left-hand-side of \eqref{leray-hirsch} is the graded $S_r \times S_{n-r}$-module
$R_{n,k}^{(r)}$.
The tensor factor $H^{\bullet}(F; \QQ)$ is the classical coinvariant module $R_r$
for the symmetric group $S_r$.
Determining the graded $S_r \times S_{n-r}$-isomorphism type of $R_{n,k}^{(r)}$ therefore reduces to 
determining the graded $S_{n-r}$-structure of $H^{\bullet}(Y_{n,k}^{(r)}; \QQ)$.

\section*{Acknowledgements}
B. Rhoades was partially supported by NSF Grant DMS-1500838.
A. T. Wilson was partially supported by an NSF Mathematical Sciences Postdoctoral
Research Fellowship. The authors thank an anonymous referee for their careful reading of the 
paper and, in particular, for pointing out the applicability of the Leray-Hirsch Theorem.
We thank Jeff Remmel for his mathematics, mentoring, and friendship.

\end{document}